\documentclass{amsart}
\usepackage{amsmath, amssymb,epic,graphicx,mathrsfs,enumerate}
\usepackage[all]{xy}

\usepackage{amsthm}
\usepackage{amssymb}
\usepackage{latexsym}
\usepackage{longtable}
\usepackage{epsfig}\normalfont
\usepackage{amsmath}
\usepackage{hhline}

 \DeclareMathOperator{\soc}{soc}
\DeclareMathOperator{\aut}{Aut}

\DeclareMathOperator{\End}{End}

\newcommand{\N}{\mathbb N}

\newtheorem{thm}{Theorem}[section]
\newtheorem{cor}[thm]{Corollary}
 \newtheorem{lemma}[thm]{Lemma}
\newtheorem{prop}[thm]{Proposition}

\numberwithin{equation}{section}

\renewcommand{\footnote}{\endnote}
\newcommand{\ignore}[1]{}\makeglossary

\begin{document}
	\bibliographystyle{amsplain}
	 %Subgroup theorems; subgroup growth, Other Dirichlet series and zeta functions, Probabilistic methods in group theory 
\title[Probabilistic zeta function]{Profinite groups in which the probabilistic zeta function has no negative coefficients}

\author{Eloisa Detomi}
\address{
Eloisa Detomi\\ Universit\`a degli Studi di Padova\\  Dipartimento di Ingegneria dell'Informazione \\email: eloisa.detomi@unipd.it}
%Dipartimento di Ingegneria dell'Informazione,   Universit\`a  degli Studi di Padova,  Via Gradenigo 6/b, 
%  35131 Padova\\ Italy 

\author{Andrea Lucchini}
\address{
Andrea Lucchini\\ Universit\`a degli Studi di Padova\\  Dipartimento di Matematica \lq\lq Tullio Levi-Civita\rq\rq \\email: lucchini@math.unipd.it}

\begin{abstract} 
 To a finitely generated profinite group $G$, a formal Dirichlet series $P_G(s)=\sum_{n \in \mathbb N} {a_n(G)}/{n^s}$ is associated, where $a_n(G)=\sum_{|G:H|=n}\mu(H, G)$ and 
  $\mu(H,G)$ denotes the M\"obius function of the lattice of open subgroups of $G.$ Its formal inverse $P_G^{-1}(s)$ is the probabilistic zeta function of $G$. 
  When $G$ is prosoluble,  every coefficient of $(P_G(s))^{-1}$ is nonnegative. In this paper we discuss the general case and we produce % existence of a non-prosoluble example and   We construct
   a non-prosoluble finitely generated group $G$  with the same property. %such that every coefficient of $(P_G(s))^{-1}$ is nonnegative.
 \end{abstract}
\maketitle

\section{Introduction} 
Let $G$ be a finitely generated profinite group, that is a
profinite group {\it topologically} generated by a finite number
of elements. For each positive integer $n$,  the number of
open subgroups of index $n$  in  $G$ is finite. So we can define a
formal Dirichlet series $P_{G}(s)$ as follows:
$$P_G(s)=\sum_{n \in \mathbb N}\frac
{a_n(G)}{n^s}\quad \quad {\mbox{ with }}\quad
a_n(G)=\sum_{|G:H|=n}\mu(H, G)$$
 where $\mu(H,G)$ denotes the M\"obius function of the lattice of open subgroups of $G,$ defined recursively by 
 $ \mu(G,G)=1$ and  $\sum_{H\leq K}\mu(K,G)=0$ for any proper open subgroup $H$ of $G$. 

The formal inverse of $P_G(s)$ is the probabilistic zeta function which was first introduced by A. Mann \cite{PFG} for finitely generated groups and by  N. Boston \cite{boston} in the case of finite groups. 
 Hall \cite{hall} showed that for   a finite group $G$ and a positive integer  $t$, $P_G(t)$ is equal to the probability that $t$ randomly chosen elements generate $G$. In  \cite{PFG} A. Mann made a conjecture which implies that $P_G(s)$ has a similar meaning for positively finitely generated (PFG) profinite groups: here we say that a profinite group is PFG if there exists a positive integer $t$ such that $\mu(\Omega(t))>0$ where $\mu$ is  the normalised Haar measure uniquely defined on $G^t$  and  $\Omega(t)$ is the set of the generating $t$-tuples  in $G$ (in the topological sense). 
 Namely Mann conjectured  that if $G$ is PFG, then the Dirichlet series $P_G(s)$ is absolutely convergent on some complex half-plane 
 %defines an analytic function on some right half-plane
  of $\mathbb{C}$ and 
    takes the values $\mu(\Omega(t))>0$  for (sufficiently large) integers $t$.   This conjecture was proved true for some  classes of profinite groups, among which finitely generated prosolvable groups \cite{A:sol}  and more generally, any PFG group $G$ with the property that,  for each open normal subgroup $N$ of $G$, all the composition factors of $G/N$ are either abelian or alternating groups \cite{A:PBMN} (see also \cite{Mann:arith} and \cite{A:X} for other classes).
 But even when the convergence of $P_G(s)$ is not ensured, this  Dirichlet series  encodes information about the lattice generated by the open subgroups of $G$   and combinatorial properties of the sequence $\{ a_n(G) \}$ reflect on the structure of $G$. For example,  a finitely generated profinite group is prosoluble if and only if the sequence $\{ a_n(G) \}$ in multiplicative (see  \cite{sol-fin} and \cite{sol}).  

In \cite{revers} the authors examine profinite groups in which %analyze the case where
 the probabilistic zeta function coincides with the subgroup zeta function 
$\zeta_G(s)=\sum_{n \in \mathbb N}\frac
{\beta_n(G)}{n^s}$ where $\beta_n(G)$ is the number of subgroups of index $n$ in $G$.  If this is the case, the  probabilistic zeta function $(P_G(s))^{-1}$ 
 is a Dirichlet series whose coefficients are  all  nonnegative. 
  For example, if $G$ is a prosoluble group,  then $P_G(s)$ can be written as $ \prod_{i \in \mathbb{N}}(1-c_i/q_i^s)$ where $q_i$ is a prime power and $c_i \ge 0$, hence all coefficients of $(P_G(s))^{-1}$ are nonnegative.   In \cite{revers}  the authors ask if there exists non-prosoluble examples of finitely generated groups for which the probabilistic zeta function is a Dirichlet series with nonnegative coefficients. 

In this paper we answer in positive to this question. 
 \begin{thm}\label{th:main}
 There exists a  non-prosoluble group $G$ such that $G$ is finitely generated and every coefficient of $(P_G(s))^{-1}$ is nonnegative. 
 \end{thm}

Note that the group constructed in Theorem \ref{th:main} is a direct product of alternating groups. This group is finitely generated but, possibly, not PFG. 
 
At the present we do not have example of a finite non-soluble group whose  probabilistic zeta function has nonnegative coefficients. 
 All the nonabelian finite  simple groups that we examined have a negative coefficient in their   probabilistic zeta function (see  Section \ref{simple}).
 For example, for the alternating group of degree $5$ we have
 \[P_{A_5}(s) =1-5/5^s -6/6^s-10/10^s+20/20^s +60/30^s-60/60^s  \] 
 and the  probabilistic zeta function $(P_{A_5}(s))^{-1}$ has a negative coefficient for $n=20$ (namely $-20$). 
 For any group $G$, the series $P_G(s)$ can be written as an (infinite formal) product of finite Dirichlet polynomials according to a chief series of $G$ as shown in \cite{crown} and in \cite{crown-profin}. Using this factorization, it is possible to construct  large finite groups in which all first terms of the 
   probabilistic zeta function are positive. For example, 
 in the last section we show that  the  probabilistic zeta function of $G=C_2^2  \times C_5^2 \times A_5$ has  negative coefficients and the first one appears  for $n=50000$. Trying to construct extension of this group by adding abelian chief factors is not a good strategy because the number of generators of the new group might be quite larger (for example, the minimal number of generators of $C_2^2 \times C_5^2 \times G$ is $4$). We will show that  
   it is possible to construct an inverse system 
  of groups $G_k$ starting from a (large enough) alternating group and extending it with direct products of alternating groups in such a way that at each step the \lq\lq first" coefficients   $(P_{G_k}(s))^{-1}$ are nonnegative   and the minimal number of generators of   $G_k$ is bounded. Hence  the resulting inverse limit $\varprojlim G_k$  is actually finitely generated 
  and its  probabilistic zeta function has no negative coefficients. 
 
 \section{Preliminaries} 
 
 In this paper we are mainly interested in finitely generated profinite groups, so,
unless stated otherwise,  \lq\lq group''  means  finitely generated profinite group and   \lq\lq subgroups''
  means closed subgroups. Moreover, in a profinite group $G$, open subgroups have finite index in $G$, and  $G$ is (topologically) $d$-generated if and only if   it is the inverse limit of $d$-generated finite groups. 
  
  We define the M\"obius function of the lattice of open subgroups of $G$  by 
 $ \mu(G,G)=1$ and  
 \[ \sum_{H\leq K}\mu(K,G)=0 \] 
 for any proper open subgroup $H$ of $G$. 
    Then we define 
formal Dirichlet series 
$$P_G(s)=\sum_{n \in \mathbb N}\frac{a_n(G)}{n^s}$$ 
as the Dirichlet  generating function associated with the sequence 
$$a_n(G)=\sum_{|G:H|=n}\mu(H,G).$$

In \cite[Theorem 13]{crown-profin} the authors show that $P_G(s)$ factorizes through a normal subgroup $N \le G$, namely $P_G(s)=P_{G/N}(s) P_{G,N}(s),$ where  
$$P_{G,N}(s)=\sum_{n \in \mathbb N}\frac
{a_n(G,N)}{n^s}\quad \quad {\mbox{ with }}\quad
a_n(G,N):=\sum_{\substack{|G:H|=n\\ HN=G}}\mu(H,G).$$ 
Furthermore, by iteration they obtain that given a chief series 
$\Sigma: G=N_0 \unrhd N_1 \unrhd \dots \unrhd  G_{\tau}=1  $, where $\tau \le \aleph_0$, 
  the  Dirichlet series 
$P_G(s)$ can be written as a formal product 
\[ P_{G}(s)= \prod_{i \ge 0} P_{G/N_{i+1},N_i/N_{i+1}}(s),\]
and they shows that the factorization is independent of $\Sigma$ \cite[Theorem 17]{crown-profin}. 

To describe the finite Dirichlet series  $P_{G/N_{i+1},N_i/N_{i+1}}(s)$ we need %to introduce
 the concept of crowns. 
 This notion was first introduced by  Gasch\"utz in \cite{g3} for chief factors of  a finite
soluble group $G$, and later generalized
to all finite groups (see for example \cite{lafuente} and
\cite{paz}). In \cite{crown} the notion of crown have been applied to study some
properties of the probabilistic zeta function of a finite group, and in \cite{crown-profin} extended to profinite groups. 
 We refer to \cite{crown-profin} for the details, but 
 for short we say that two chief factors of  $G$ are $G$-equivalent (as $G$-groups) if they are $G$-isomorphic either between them or 
 to the two minimal normal subgroups of a finite primitive epimorphic image of $G$ (having two minimal normal subgroups). In particular, two $G$-equivalent factors are equivalent as groups. 
  Given a chief series $\Sigma$ %=\{N_i\}$
   of $G$, for any  chief factor $A$ we set $\delta_G(A)$ to be the cardinality of the set of the chief factors $N_i/N_{i+1}$ which are not Frattini and are $G$-equivalent to $A$; this number does not depend on the chosen chief series and has the remarkable property that $G$ has  a section $G$-isomorphic to $A^{\delta_G(A)}$ (called the $A$-crown of $G$). 
  For example, if $G= \prod_i A_{m_i}^{f_i}$   is a direct product of copies of alternating groups, 
  for some integers $f_i$ and $m_i \neq m_j$, then $\delta_G(A_{m_i}) =  f_i$. 
   
 Let $A$ be  a chief factor of $G$ %an   irreducible $G$-group 
  and let $\rho:G \rightarrow \aut(A)$ be defined by $g\mapsto
g^{\rho}$, where $g^{\rho}:a \mapsto a^{g}$ for all $a \in A$.
The \emph{monolithic primitive group associated with $A$}  is defined as
$$L_A=
  \begin{cases}
    G^{\rho}  A \cong(G/C_G(A)) A& \text{ if $A$ is abelian}, \\
    G^{\rho} \cong G/C_G(A)& \text{ otherwise}, 
  \end{cases}
$$ 
 and we identify $\soc(L_A)$ with $A$. 
Then we can  write $P_G(s)$ as a formal product 
\begin{equation}\label{**}
%P_G(s)=\prod_A \left(\prod_{ 1 \leq i \leq \delta_G(A)} P_{L,A}(s) - \frac{(1+q+\dots+q^{i-2})\gamma}{|A|^s}   \right)
P_G(s)=\prod_A \left(  \left( P_{L_A,A}(s) \right) \cdots \left(P_{L_A,A}(s) - \frac{(1+q+\dots+q^{\delta_G(A)-2})\gamma}{|A|^s}   \right) \right)
\end{equation}
 where $A$ runs over the set of representatives (under the $G$-equivalence relation) of the non-Frattini chief factors in a given chief series $\Sigma$,  
 %irreducible $G$-groups $G$-equivalent to a non-Frattini chief factor of $G$, 
 $L_A$ is the monolithic primitive group associated with $A$, $\gamma =|C_{\aut A} (L_A/A)|$ and  $q=|\End_{L_A} A|$ if $A$ is
abelian, $q=1$ otherwise. 
Moreover, the factorization of $P_G(s)$ does not depend on the chosen chief series $\Sigma$. 
%corresponding to the non-Frattini factors in a chief series $\Sigma$ of $G$ is precisely $(\ref{**})$, independently of the choice of $\Sigma$.
In particular, if $G$ is a prosoluble group, then 
\[ P_G(s) = \prod_{q_i} \left( 1-\frac{c_i}{q_i^s} \right) \] 
for some positive integer $c_i$ and some prime powers $q_i$. 

In the case where $G$ is a finite group and $G = H \times K$, where $H$ and $K$ have no common isomorphic chief factors, it is possible to prove that 
 the formula \ref{**} reduces to the result of Brown \cite{brown}
 \begin{equation}\label{eq:brown} P_G(s) = P_H(s) P_K(s) 
 \end{equation}
  and the minimal number of generators $d(G)$ of $G$ is the minimum between the minimal number of generator of $H$ and of $K$ (see \cite{dv-l} for the complete  result on $d(G)$).  
Moreover, if $G=S^f$  for a non-abelian finite simple group $S$ and an integer $f$, then the formula \ref{**} reduces to the result of Boston \cite{boston}
 \begin{equation}\label{eq:boston} 
 P_{S^f}(s) = \prod_{i=0}^{f-1} \left( P_S(s) - \frac{i | \aut(S)| }{|S|^s} \right).
 \end{equation}

Let us denote by $c_n(G)$ the coefficients of the probabilistic zeta function of $G$: 
  $$(P_{G}(s))^{-1} =\sum_{n \ge 1} \frac{c_n(G)}{n^s}.$$
 Basic properties of the coefficients $a_n(G)$ and $c_n(G)$ are summarised in the following lemmas.
\begin{lemma}\label{c-base} 
Given a group $G$, the following relations hold:  
\begin{enumerate}
\item $ c_n(G) = - \sum_{\substack{rs=n \\ r\neq 1}} a_r(G)c_s(G);$
\item $ |c_n(G) |\le   \sum_{n_1 \cdots n_t=n} |a_{n_1}(G) \dots a_{n_t}(G)|$, where the sum runs over the set of  all ordered factorizations of $n.$
\end{enumerate}
 \end{lemma}
 
 \begin{lemma}\label{n-divide-c_n}
If $G$ is a finite perfect group, then $n$ divides $a_n(G)$ and $c_n(G)$ for every $n$. 
 \end{lemma}
\begin{proof} 
Let $H$ be a subgroup of $G$ of index $n$. 
By \cite[Theorem 4.5]{how}, the index $|N_G(H):H|$ divides $\mu(H,G)|G:HG'|$. Since $G$ is perfect, it follows that $n=|G:H|$ divides $\mu(H,G)|G: N_G(H)|$. So, in particular $n$ divides $a_n(G)$. By Lemma \ref{c-base}, we deduce that $n$ divides also $c_n(G)$. 
\end{proof}
 
  Recall that if $H$ is an open subgroup of $G$ and $\mu(H,G) \neq 0$, then $H$ is an intersection of maximal subgroups of $G$. 
 We are interested in subgroups that give a nontrivial contributions to $a_n(G)$, 
% Let us write 
  so we set %use the following notation:  
  $$b_n(G) = |\{H \le G \mid \mu(H,G) \neq 0, \ |G:H|=n\}|.$$
 The next two results will be the key ingredient to construct our group.  
 \begin{thm}\cite[Theorems 1-2]{vale}\label{valentina}
 There exist two absolute constants $\alpha$  and $\beta$ such that for any $m \in \N$, if $G$ is an alternating or symmetric group, % of degree $m$, 
 then 
 \begin{enumerate}
 \item $b_n(G) \le n^{\alpha}$, for every integer $n$. 
 \item $|\mu(H,G)| \le |G:H|^{\beta}$ for every subgroup $H$ of $G$. 
 \end{enumerate}
 \end{thm}
 
  \begin{thm}\cite[Theorem 9]{nontrivialM}\label{andrea-nontriv}
  Suppose that $G$ is a $d$-generated profinite group, and that there exists a
constant $a$ with the following property: for any epimorphic image $L$ of $G$ which is monolithic with non-abelian socle and for any 
 $X \le  L$
 $$| \mu(X,L)| \le |\mu(X\soc(L), L)| \cdot |X\soc(L):X|^a.$$
 Then
 $$| \mu(H,G)| \le  |G: H|^{\tilde{a}}$$
for each open subgroup $H$ of $G$, where 
$\tilde{a}=\max(a+1, d+1)$. 
 \end{thm}

\section{Product of alternating groups}

Let 
$c = \alpha+\beta+11$ where $\alpha$ and $\beta$ are the constants defined in Theorem \ref{valentina}. Assume that $f$ is a function from $\mathbb N$ to $\mathbb N$ with the property that 
$$f(m) \le (m!)^c,$$ and let  $\{m_i \mid i\in \mathbb N\}$ be a strictly increasing sequence of integers.
 
 Consider the product of $f(m_i)$ copies of the alternating groups $A_{m_i}$, for $i=1, \dots , k$:
% group 
 $$G_k= \prod_{i=1}^{k} A_{m_i}^{f(m_i)}.$$

Our first goal is  to prove that there exists an integer $N$ such that if $m_1\ge N$, then $G_k$ is $(c+2)$-generated and $|c_n(G_k)| \le (n!)^c$ for every  integer $n>m_k$. 
 
The proof that $G_k$ is boundedly generated, once $m_1$ is large enough, 
relies on the following result; throughout, all logarithms are to base 2. 

\begin{thm}\cite{nina}\label{colva} 
Let $S$ be a  nonabelian finite simple group. Let $h_S(d)$ be the maximal integer $h$ such that $S^h$ is $d$-generated. Then 
$$h_S(d) \ge \frac{|S|^{d-1}}{\log{|S|}},$$
 for all $d \ge 2$. 
\end{thm}

\begin{prop}\label{prop:d}
Assume that  $m_1!/\log(m_1!/2) \ge 2^{c+1}.$ Then  $G_k$ is $(c+2)$-generated. 
\end{prop} 
\begin{proof} 
Note that, under the given conditions, $m_1 \ge 5,$ hence each $A_{m_i}$ is a   nonabelian finite simple group. 
Since $G_k$ is the direct product of the  groups $A_{m_i}^{f(m_i)}$, with $m_i \neq m_j$ for $i \neq j$, it is sufficient to prove that 
$A_m^{f(m)}$ is $(c+2)$-generated for every $m \ge m_1$. 

Theorem \ref{colva}  states that 
the maximal integer $h=h_{A_m}(c+2)$ such that $A_m^{h}$ is $(c+2)$-generated satisfies the bound 
$
%h_{A_m}(c+2) 
h\ge  (m!/2)^{c+1}/ \log{(m!/2)}  $. 
As $m_1!/\log(m_1!/2) \ge 2^{c+1}$ and $m \ge m_1$,  
  we have that $m!/ \log{(m!/2)}  \ge 2^{c+1}$.  Hence 
 $$%h_{A_m}(c+2)
 h \ge \frac{(m!/2)^{c+1}}{\log{(m!/2)}} \ge (m!)^{c}. $$ 
 Since $f(m) \le (m!)^c \le h, $  we conclude that  
 $A_m^{f(m)}$ is $(c+2)$-generated. 
\end{proof}

\begin{cor}\label{cor:mu}
%If  $G_k$ is $d$-generated and $d \ge \beta$, then 
Assume that  $m_1!/\log(m_1!/2) \ge 2^{c+1}.$ 
Then 
  $$| \mu(H,G_k)| \le  |G_k: H|^{c+3}$$
for each subgroup $H$ of $G_k$
\end{cor} 
\begin{proof} By Proposition \ref{prop:d} we know that $G_k$ is $(c+2)$-generated. 
Each monolithic image $L$ of $G_k$ is isomorphic to $A_{m_i}$, for some $i$, and $L=\soc (L)$. By   Theorem \ref{valentina},  for any 
 $X \le  L$ 
   $$| \mu(X,L)| \le |L:X|^{\beta}  = |\mu(X\soc(L), L)| \cdot |X\soc(L):X|^{\beta}.$$
  We can thus apply Theorem  \ref{andrea-nontriv} to deduce that for each open subgroup $H$ of $G_k$ 
    $$| \mu(H,G_k)| \le  |G_k: H|^{\tilde{a}},$$
 where 
$\tilde{a}=\max(\beta+1, (c+2)+1)=c+3$, as $c+2 \ge \beta$. 
%Straightforward from Theorems \ref{valentina} and \ref{andrea-nontriv}, since $d=c+2 \ge \beta$. 
\end{proof}

\begin{lemma}\label{lem:zero}
Let $n > 5$ be an integer and let  $n_1 \dots n_t=n$ be a  nontrivial  (positive) factorization  of $n$, that is $t \neq 1$ and $n_i > 1$ for every $i$. 
 Then 
\[  \prod_{i=1}^t n_i! \le \frac{n!}{n^2}.\] 
\end{lemma}
%\begin{proof}
%If $n=2^a \ge 8$,  and $n_i=2$ for every $i$, then it's easy to see that $2^a \le (2^a!)/(2^a)^2$.  
%If $n_i \neq 2$ for some $i$, then the claim follows from the case $t=2$: it is sufficient to prove that if $n =rs$ with $r \ge s \ge 2$ and $r \neq 2$, then 
%\[ r! r^2 s! s^2 \le n! . \] 
%If $s=2$, then $2! r! 2r \le (2r-1)!$ follows from  $2! r! 2r= 4 r! r \le (r+2)(r+1) r! \le (2r-1)!$. 
%
%If $s \neq 2$, then  $r! r^2 s! s^2$ is a product of $r+s+4 \le 3r -1$ nontrivial integers, while $(rs)!$ is a product of $rs-1 \ge 3r -1$ (respectively) greater integers.  
%\end{proof}
 
 \begin{lemma}\label{lem:MS}\cite[Lemma 8]{MS}
 For every $\epsilon > 0$ there exists an integer $N = N(\epsilon)$ such that, if $\soc(G)$ is alternating
and $n > N(\epsilon),$ then $G$ has less than $n^{1+\epsilon}$ maximal subgroups of index $n$. 
   \end{lemma}
   
 \begin{lemma}\label{lem:max}
 There exists an integer $N \ge 5$ 
  such that if $m_1 \ge N$,  then  for every integer $n $ there 
 are at most $$n^3 (n!)^c$$
 maximal subgroups of index $n$ in $G_k. $
 \end{lemma}
 \begin{proof} 
 By Lemma \ref{lem:MS}, applied with $\epsilon =1/2$, there exists an integer $N(1/2)$ such that if $n \ge N(1/2)$, then an alternating group has    less than $n^{1+1/2}$ maximal subgroups of index $n$. 
 We set $N$ to be the minimal integer greater than $N(1/2)$ and satisfying the condition $N!/\log(N/2) \ge 2^{c+1}.$
 Note that for $m_1 \ge N$, 
  the index of any subgroup of $G_k$ is always at least $N$. 
 
 Let $M$ be a maximal subgroup of index $n$ in $G_k$. Since $G_k$ is a direct product of nonabelian finite simple groups  (namely $A_{m}$, for some $m \ge m_1$), 
 then  one of the following holds  (see e.g. \cite[Lemma 16]{PFG}): 
\begin{itemize} 
\item[1)] $M=T_i \times K,$ 
where $T_i$ is the product of all minimal normal subgroups of $G_k$ but one,  isomorphic to some $A_m$, and $K$ is a maximal subgroup of $A_m$, 
\item[2)] $M=T_{i,j} \times D,$
 where $T_{i,j}$ is the  product of all minimal normal subgroups of $G_k$ but two  isomorphic to $A_m$ and $D$ is a \lq\lq diagonal subgroup" of $A_m \times A_m$, i.e.   
  there exists  and an automorphism $\phi$ of $A_m$ such that $D = \{(x,x^\phi) \mid x \in A_m\}$. 
  \end{itemize}

  In  case (1), we have at most $f(m)$ choices for $A_m$. 
Moreover,  by the definition of $N$,  
   there are less then   
    $n^{1+1/2}$ choices for the maximal subgroup $K$ of index $n$ in $A_m$. Thus we have at most 
  $$ \sum_{m \le n}  n^{1+1/2}f(m) \le \sum_{m \le n}  n^{1+1/2} (m!)^c \le n^{2+1/2} (n!)^c$$
  maximal subgroups $M$ of the first type, since $m \le n$. 
  
  In case (2), we notice  that $n=|G_k:M|=|A_m|=m!/2$, hence $m!=2n$.  We have at most $f(m)(f(m)-1)/2$ choices for $T_{i,j}$, and once $T_{i,j}$ is fixed, we have $|\aut(A_m)| \le 2m!=4n$ choices for the automorphism $\phi$.  Since $f(m) \le (m!)^c$, we have at most 
  \[ \frac{f(m)(f(m)-1)}{2} \cdot 2m! \le (2n)^c ((2n)^c-1) 2n \le (2n)^{2c+1}\] 
   maximal subgroups $M$ of this type. 
 Actually, since $ m \ge 5$, we have this type of maximal subgroups only for $n \ge 5!/2$, 
 and thus we have at most  
 $$(2n)^{2c+1} \le \frac{ (n!)^c}{n}$$ 
 of these subgroups. 
 Adding the two bounds,  we conclude that there are  at most 
 \[ n^{2+1/2} (n!)^c + \frac{ (n!)^c}{n} \le n^3 (n!)^c \] 
  maximal subgroups of index $n$ in $G_k. $
  \end{proof}

 \begin{prop}\label{prop:a_n}
 There exists an integer $N \ge 5$ such that if $m_1 \ge N$,    then  
 $$|a_n(G_k) | \le n^{c}(n!)^c$$
 for every integer $n$. 
 In particular, if $n > m_k$, then 
  $$|a_n(G_k) | \le \frac{3}{n} (n!)^c.$$
 \end{prop}
 \begin{proof}
 %If $H$ is a  maximal subgroup of $G_k$, then $\mu(H,G_k)=-1$, thus 
 %$$ |a_n(G_k)| \le |\sum_{H \textrm{non-max}\\ |G_k:H|=n } \mu(H, G_k)|. $$ 
 % For $n \ge 5$, we want to count the number of non-maximal subgroups
 Let $N$ be the integer defined in Lemma \ref{lem:max} and let  $\Omega$  be the set of   subgroups $H$  of index $n$ in $G_k$ such that $\mu(H,G_k) \neq 0$. 
 If $H \in \Omega$, then,   %is a non-maximal subgroup  of index $n$ in $G_k$ such that $\mu(H,G_k) \neq 0$, 
 by \cite[Theorem 1]{nontrivialM},  there exist a factorization $n=n_1 \cdots n_t$ (with $n_i>1$) and a family of subgroups $Y_1, \dots , Y_t$   of $G_k$ 
 satisfying the following properties: 
 \begin{enumerate}
 \item $H= Y_1 \cap \dots \cap Y_t$;
 \item $|G_k:Y_i|=n_i$; 
 \item $\mu(Y_i,G_k) \neq 0$ for every $i$; 
 \item either $Y_i$ is a maximal subgroup of $G_k$ or there exists a normal subgroup $K_i$ of $G_k$ such that  $K_i \le Y_i$ and $G_k/K_i$ is simple.
 \end{enumerate} 
 Let $\Omega_1$ be the set of subgroups $H \in \Omega $ such that $t\neq 1$, 
  let $\Omega_2$ be  the set of subgroups $H \in \Omega$ % \setminus \Omega_1$
   such that  there exists a normal subgroup $K$ of $G_k$ such that  $K\le H$ and $G/K$ is simple,  and let $\Omega_3 = \Omega \setminus (\Omega_1 \cup \Omega_2)$. 
 
 Note that if $H \in \Omega_2$, then there exists a normal subgroup $K$ of $G_k$ such that  $K\le H$ and $G/K \cong A_m$ for an integer $m \le \min(n,m_k)$. 
For each $m \le \min(n,m_k)$, we have at most $f(m)$ choices for $K$ and, once $K$ is fixed, by  Theorem \ref{valentina}, we have at most $b_n(A_m) \le n^{\alpha}$ choices for $H/K$. Moreover, $\mu(H,G_k) = \mu (H/K, G_k/K) \le n^{\beta}$. 
Thus 
\begin{eqnarray}\label{2} 
\left|\sum_{H  \in \Omega_2} \mu(H,G_k) \right| \le \sum_{m \le n} \left(f(m) n^{\alpha} n^{\beta}\right) \le n^{\alpha+\beta+1} (\min(n,m_k)!)^c.
 \end{eqnarray} 
 
 If  $H \in \Omega_3$, then $H$ is a maximal subgroup  of $G_k$, hence 
 $\mu(H,G_k)=-1$. In the proof of Lemma \ref{lem:max} we have seen that $G_k$ has at most $(n!)^c/n$ of this kind of maximal subgroups, 
 thus 
 \begin{eqnarray}\label{3} 
\left|\sum_{H  \in \Omega_3} \mu(H,G_k) \right| \le \frac{(n!)^c}{n}.
 \end{eqnarray} 
 
 Let now $H \in \Omega_1$. Then  there exists a nontrivial factorization $n=n_1 \cdots n_t$ and a family of subgroups  $Y_1, \dots , Y_t$  of $G_k$ with $t \neq 1$ 
 satisfying the above properties.  
 There are at most $n^2$ possible choices for the factorization $n=n_1 \cdots n_t$  (see \cite{partitions}): fix one of them. By Lemma \ref{lem:max} we have at most $n_i^3 (n_i!)^c$ choices of $Y_i$ if $Y_i$ is maximal. If $Y_i$ is not  maximal, then  $Y_i/K_i$ is a subgroup of $G_k/K_i \cong A_{m}$, for some $m\le n_i$, with nontrivial  value of the 
  M\"obius   function. 
 For each $m \le n_i$, we have at most $f(m) \le (m!)^c$ choices for $K_i$ and then, by  Theorem \ref{valentina}, at most $b(n_i) \le n_i^{\alpha}$ choices for $Y_i/K_i$. 
  Therefore we have at most 
  \[  n_i^3 (n_i!)^c + \sum_{m=m_1}^{n_i} n_i^{\alpha} f(m) \le n_i^3 (n_i!)^c + n_i \cdot n_i^{\alpha} (n_i!)^c \le n_i^{\alpha+5} (n_i!)^c\]
 choices for  $Y_i$. 
 Thus, applying  Lemma \ref{lem:zero} we get  
  \begin{eqnarray*} 
 | \Omega_1| % & \le& \sum_{n_1 \cdots n_t=n, \ t\neq 1} \prod_{i=1}^t n_i^{\alpha+1} f(n_i)\\
  & \le& \sum_{\substack{n_1 \cdots n_t=n\\ t\neq 1}} \left( \prod_{i=1}^t n_i^{\alpha+5} (n_i!)^c \right)\\
  & \le& n^2 n^{\alpha+5} \left( \frac{n!}{n^2} \right)^c
  = n^{\alpha+7-2c}(n!)^c.
 \end{eqnarray*}  
By Corollary \ref{cor:mu},   $ \mu(H, G_k) \le n^{c+3}$ for every subgroup $H$ of $G_k$, hence 
 \begin{eqnarray}\label{1} 
  \left|\sum_{H \in \Omega_1 } \mu(H, G_k) \right| \le   | \Omega_1| n^{c+3}
 \le  n^{\alpha+10- c} (n!)^c.  
 \end{eqnarray}  
 
 Combining (\ref{2}), (\ref{3}) and  (\ref{1}) and taking into account that $c \ge \alpha+\beta+ 11$, we deduce that 
   \begin{eqnarray*} 
    |a_n(G_k)| &\le&  n^{\alpha+\beta+1} (\min(n,m_k)!)^c + n^{\alpha+10- c} (n!)^c +\frac{(n!)^c}{n}  \\
& \le&  n^{c} (n!)^c . 
  \end{eqnarray*} 
 %whenever $c+\beta \ge 4$. 
 
 Whenever  $n> m_k$, a sharper bound holds: indeed $  \min(n,m_k)! = m_k!\le (n-1)! =n!/n$ hence from 
  (\ref{2}), (\ref{3}) and  (\ref{1})  we conclude that 
   \begin{eqnarray*} 
  |a_n(G_k)|  %&\le&  n^{\alpha+\beta+1} (\min(n,m_k)!)^c + n^{\alpha+10- c} (n!)^c +\frac{(n!)^c}{n}\\ 
  &\le&   n^{\alpha+\beta+1} \left(\frac{n!}{n}\right)^c + n^{\alpha+10- c} (n!)^c +\frac{(n!)^c}{n}\\
&  = &  n^{\alpha+\beta+1 -c} (n!)^c + n^{\alpha+10- c} (n!)^c +\frac{(n!)^c}{n}
 \\ & \le&   \frac{3}{n}(n!)^c 
  \end{eqnarray*} 
  as claimed. 
 \end{proof}

\begin{prop}\label{prop:c_n}
Under the assumption of Proposition \ref{prop:a_n} and for any $n >m_k$ 
$$|c_n(G_k)| \le (n!)^c.$$
 \end{prop}
 \begin{proof}
 By Lemma \ref{c-base} 
 $$ |c_n(G_k)| \le \sum_{n_1 \cdots n_t=n} |a_{n_1}(G_k) \dots a_{n_t}(G_k)|,$$
 where the sum runs over the set of all ordered factorizations  $n_1 \dots n_t=n$   with  $n_i > 1$. 
 Notice that number $H(n)$ of ways to factor a natural number $n$ into an ordered product of integers,
 each factor greater than one, is at most $n^{\zeta^{-1}(2)}\leq n^2$ \cite{orfa}. 
 As $n > m_k$, it follows from Proposition \ref{prop:a_n} that  $ |a_n(G_k)| \le  \frac{3}{n}(n!)^c$, while the bound 
   $|a_{n_i}(G_k) | \le n_i^{c} (n_i!)^c $ suffices for  any nontrivial factorization $n_1 \cdots n_t=n$ of $n$.  
Indeed, applying Lemma \ref{lem:zero}, we get 
 \begin{eqnarray*} 
 |c_n(G_k)| % &=&  \sum_{n_1 \cdots n_t=n} |a_{n_1}(G_k) \dots a_{n_t}(G_k)| \\
  &\le&  \sum_{\substack{n_1 \cdots n_t=n\\ t\neq 1}} \left( \prod_{i=1}^t  |a_{n_i}(G_k)| \right) + |a_n (G_k)|\\
 &\le&  \sum_{\substack{n_1 \cdots n_t=n\\ t\neq 1}}  \left(  \prod_{i=1}^t   n_i^{c} (n_i!)^c \right) +  \frac{3}{n}(n!)^c\\
 &\le & H(n) \cdot n^{c} \left(\frac{n!}{n^{2}}\right)^c+ \frac{3}{n}(n!)^c\\
% &\le &  n^2 n^{c} \left(\frac{n!}{n^{2}}\right)^c+ \frac{3}{n}(n!)^c\\ 
  %&\le & n^{c+2-2c} (n!)^c + \frac{3}{n}(n!)^c\\\
  &\le &   (n!)^c, 
  \end{eqnarray*}  
  where the last inequality follows from the fact that $ c \ge 3$.
 \end{proof}

Now we are ready to prove Theorem  \ref{th:main} by constructing   a non-prosoluble  finitely generated group $G$ such that  every coefficient of $(P_G(s))^{-1}$ is nonnegative.
% \begin{thm}\label{th:main}
% There exists a profinite non-prosoluble group $G$ such that $G$ is finitely generated and every coefficient of $(P_G(s))^{-1}$ is nonnegative.
% \end{thm}
 \begin{proof}[Proof of Theorem \ref{th:main}]
 Let $c = \alpha+\beta+11$, where $\alpha$ and $\beta$ are the constants defined in Theorem \ref{valentina}, and let $N$ be the constant defined in 
  Proposition \ref{prop:a_n}.
 
 We set $m_1=N$ and we note that, by the definition of $N$,  $m_1/\log(m_1/2) \ge 2^{c+1}.$
 Then we set $f_1=1$ 
  and $G_1=A_{m_1}$. 
The probabilistic zeta function of $G_1=A_{m_1}$ has some negative coefficients (see Section \ref{simple}):  
 let $m_2$ the first integer such that $c_{m_2}(G_1)<0$  and set 
 $$f_2= - \frac{c_{m_2}(G_1)}{m_2} \quad \textrm{and}  \quad G_2=  A_{m_1} \times A_{m_2}^{f_2}. $$ 
Note that, by Proposition \ref{prop:c_n},  $f_2 \le (m_2!)^c$ and,  by Lemma \ref{n-divide-c_n},  $m_{2}$ divides $c_{m_2}(G_1)$, hence $f_2$ is actually  an integer.   

Assume that we have defined two sequences of integers $m_1< m_2 < \dots < m_{k}$ and $f_1, f_2, \dots , f_k, $ 
 such that $f_i  \le (m_i !)^c$ and the coefficients of the probabilistic zeta function of $G_k = \prod_{i=1}^{k} A_{m_i}^{f_i}$  
 satisfy 
 $$ c_{n}(G_k) \ge 0, \quad \forall \ n \le m_k.$$ 
  If all coefficients $c_{n}(G_k)$ are nonnegative, then we set $G=G_k$, and we are finished.  
 Otherwise,  we define $m_{k+1}$ to be the first integer $n$ such that $c_{n}(G_k) <0 $. By Lemma \ref{n-divide-c_n},  $m_{k+1}$ divides $c_{m_{k+1}}(G_k)$. 
 So, we set 
 $f_{k+1}=   {-c_{m_{k+1}}(G_k)}/{m_{k+1}}$ and 
  $$ G_{k+1}=  G_k \times  A_{m_{k+1}}^{f_{k+1}}= \prod_{i=1}^{k+1} A_{m_i}^{f_i}. $$ 
 Note that, since $m_{k+1} > m_k$,  from  Proposition \ref{prop:c_n} (where $f$ is a function such that $f(m_i)=f_i$)  it follows that $f_{k+1} \le (m_{k+1}!)^c$. 
   
 By   equation \ref{eq:brown} 
 \begin{equation} 
 P_{G_{k+1}}(s) = P_{G_{k}}(s) \cdot  P_{A_{m_{k+1}}^{f_{k+1}}}(s). 
 \end{equation}
Moreover, by equation \ref{eq:boston}, we can evaluate the first nontrivial terms of $ P_{A_{m_{k+1}}^{f_{k+1}}}(s)$
  \begin{eqnarray*} 
   P_{A_{m_{k+1}}^{f_{k+1}}}(s) &=& \prod_{i=0}^{ f_{k+1} -1}    \left( P_{A_{m_{k+1}}}(s) -  \frac{i | \aut(A_{m_{k+1}})| }{|A_{m_{k+1}}|^s} \right)\\
   &=&  \prod_{i=0}^{ f_{k+1} -1}     \left( 1- \frac{{m_{k+1}}}{{m_{k+1}}^s} + \cdots \right) \\
      &=&  1- \frac{ f_{k+1}{m_{k+1}}}{{m_{k+1}}^s} + \cdots  , 
 \end{eqnarray*}
 hence 
\[  P_{A_{m_{k+1}}^{f_{k+1}}}(s) ^{-1} =1+  \frac{ f_{k+1}{m_{k+1}}}{{m_{k+1}}^s} + \cdots. \] 
  Since $ f_{k+1}{m_{k+1}}={-c_{m_{k+1}}(G_k)}$, from 
\begin{eqnarray*} 
 P_{G_{k+1}}(s)^{-1} &=& P_{G_{k}}(s)^{-1}  \cdot  P_{A_{m_{k+1}}^{f_{k+1}}}(s) ^{-1} \\
 &=& \left( \sum_{n} \frac{c_n(G_k)}{n^c} \right)   \cdot   \left( 1+  \frac{ f_{k+1}{m_{k+1}}}{{m_{k+1}}^s} + \cdots  \right) 
 \end{eqnarray*}
 we deduce that $c_{m_{k+1}}(G_{k+1})=0$ and so  the coefficients $c_{n}(G_{k+1})$ are nonnegative for every $n \le m_{k+1}$. 
   This shows that the chosen  integers $m_{k+1}$ and $f_{k+1}$ satisfies the above conditions. 
   
   Let $G$ be the inverse limit of the finite groups $G_k$.    By Proposition \ref{prop:d}, each $G_{k}$   is $(c+2)$-generated. 
   Therefore $G$ is  $(c+2)$-generated. 
   
   Let $P_G(s)^{-1} = \sum_n c_n(G)/n^s$. Note that $c_n(G) = c_n(G_k)$ whenever $n \le m_k$, 
   since $P_G(s)^{-1}= P_{G_k}^{-1} (s) P_H(s)^{-1}$ for $H= \prod_{j >k} A_{m_j}^{f_j}$ and $c_{n}(H)=0$ for every $0 \neq  n \le m_{k+1}$. 
   It follows that all the coefficients $c_n(G)$ of the probabilistic zeta function of $G$ are nonnegative. 
 \end{proof}

\section{Simple groups}\label{simple}  
We know no example of a nonabelian 
 finite simple group $G$ for which the probabilistic zeta function $(P_G(s))^{-1}$ has no negative coefficients. In the case of the small simple groups for which the table of marks is available, the existence of negative coefficients  in  $(P_G(s))^{-1}$ can be easily detected using GAP. 

In the case of alternating groups,  we may use the following result: if $G=A_n$, $n\geq 9$ and $|A_n:K|\leq n(n-1),$ then $K$ is either a point-stabilizer, or a 2-set stabilizer, or the intersection of two point-stabilizers (see for example \cite[Theorem 5.2A]{dm}). This implies 
$$P_G(s)=1-\frac{n}{n^s}-\frac{n(n-1)/2}{(n(n-1)/2)^s}+\frac{n(n-1)}{(n(n-1))^s}+\dots$$ and consequently 
 $$(P_G(s))^{-1}=1+\frac{n}{n^s}+\frac{n(n-1)/2}{(n(n-1)/2)^s}-\frac{n(n-1)}{(n(n-1))^s}+\dots$$
 has a negative coefficient for $n(n-1).$
  For $5 \le n \le 8$,  we can use GAP to show the existence of negative coefficients   in  $(P_G(s))^{-1}$. 
 
 When $G$ is a simple group of Lie type defined over a field of characteristic $p$, one can consider  the inverse of the series 
 $P_G^{(p)}(s)=\sum_{(n,p)=1}a_n(G)/n^s$, which can be easily described since it depends only on the parabolic subgroups of $G$ (see  \cite[Theorem 17]{pat}). For example, it is not difficult to see that if $G\neq PSL(2,q)$ is an untwisted group of Lie type, then  $(P_G^{(p)}(s))^{-1}$ has at least a negative coefficient. 
   Thus  also $(P_G(s))^{-1}$ has a negative coefficient. 
 The case $G=PSL(2,q)$ requires a more detailed case by case analysis, but the probabilistic zeta function in  known \cite{pat2} and it is possible to see again that there is a negative coefficient.
 
 \section{The Probabilistic Zeta Function of $C_2^2  \times C_5^2 \times A_5$.} 
 
 It seems quite challenging  to  construct a possible example of a finite non-soluble group whose probabilistic zeta function has no negative coefficients. 
 To describe some of the difficulties that arises,  in this section we  sketch the 
 calculation for the "easy" group $G=C_2^2  \times C_5^2 \times A_5$;  here the polynomials  involved in the factorization of $P_G(s)$  are all known and easy to handle. 
 We will see that the first negative coefficient of  $P_G^{-1}(s)$ occurs  for $n= 50000$.

 Recall that 
 \[P_{A_5}(s) =1-5/5^s -6/6^s-10/10^s+20/20^s +60/30^s-60/60^s.  \] 
   % Let $a_i(A_5)$ be the coefficients of this series  and let $c_n(A_5)$  be the coefficients of the inverse series $(P_{A_5}(s))^{-1}$. 
   Thus, the  probabilistic zeta function $(P_{A_5}(s))^{-1}$ has a negative coefficient for $n=20$: 
       $c_{20}(A_5)=-20.$
 
  Let $G=  H \times A_5$ where  $H=C_2^2 \times C_5^2 $. 
 Note that 
 \[P_G (s)= P_{C_2^2}(s) \cdot P_{C_5^2}(s) \cdot P_{A_5}(s), \] 
 as the chief factors of $G$ in the sections %corresponding to
  $C_2^2, C_5^2 $ and $ A_5$ are not $G$-equivalent. 
 
  Moreover
% \[P_{C_p^2}(s)^{-1}= \left(\sum_{i=0}^{\infty} (1/p^s)^i \right)\left(\sum_{j=0}^{\infty} (p/p^s)^j \right) %= \sum_{n=0}^{\infty}\left( \sum_{i+j=n} 1^i p^j\right)/p^{ns}
% =\sum_{n=0}^{\infty}\left( \sum_{j=0}^n  p^j\right)/p^{ns}
 \[P_{C_p^2}(s)^{-1}= \left(\sum_{i=0}^{\infty} \left(\frac{1}{p^s} \right) ^i \right)\left(\sum_{j=0}^{\infty} \left( \frac{p}{p^s} \right) ^j \right) %= \sum_{n=0}^{\infty}\left( \sum_{i+j=n} 1^i p^j\right)/p^{ns}
 =\sum_{n=0}^{\infty} \frac{\left( \sum_{j=0}^n  p^j\right)}{p^{ns}}
  \]
Thus 
$ P_H(s)^{-1}=  \sum_{i=0}^{\infty}  c_n(H)/n^s $ 
where, for any $i $ and $k$,  
\[ c_{2^i5^k}(H)=\left( \sum_{l=0}^i  2^l\right)\left( \sum_{t=0}^k  5^t\right)= \frac{ \left(2^{i+1}-1\right) \left(5^{k+1}-1\right)}{4}\]
 and $c_n(H)=0$ if $n$ is not of the form $2^i5^k$. 
 
 As $P_G(s)=P_{H}(s) \cdot P_{A_5}(s)$ we can evaluate the coefficients $c_n(G)$ of $ P_G(s)^{-1}$ by the formula 
 \begin{eqnarray*}
 P_G(s)^{-1}\cdot P_{A_5}(s)&=& P_{H}^{-1}(s) %\\
% \left( \sum_{n=0}^{\infty}  \frac{c_n(G)}{n^s }\right)\left( \sum_{n=0}^{60} \frac{a_n(A_5)}{n^s }\right)&=&\sum_{n=0}^{\infty}  \frac{a_n(H)}{n^s}
\end{eqnarray*}
 which gives $  c_n(H) = \sum_{\substack{ur=n\\ u\neq 1}}  a_u(A_5)c_r(G)$ and thus 
 \begin{equation*}
  c_{2^i 5^k} (G)= 
 5 c_{2^i 5^{k-1}} (G)+ 10 c_{2^{i-1}5^{k-1}}(G) -  20 c_{2^{i-2}5^{k-1}}(G) +  c_{2^i 5^k}(H), 
 \end{equation*}
  for $i, k \ge 1$. 
 A straightforward calculation shows that $c_n (G)<0$  for $n= 50000$ and this is the first negative coefficient of $ P_G(s)^{-1}$. 
 
 %Actually, whenever $G=Q \times C_2^2 \times C_5^2   \times A_5$ with $(|Q|, 10) =1)$, we have a negative coefficient in $P_{G}(s)^{-1}$ for $n=50000$.
% GAP:  l(a,c)=c(2^{a-1}5^{c-1} 
% inv:=function(t)
%local l,a,c;
%l:=[];
%
%l[1]:=[];
%l[1][1]:=1;
%for c in [2..t] do
%l[1][c]:=5*l[1][c-1]+(5^c-1)/(5-1);
%od;
%
%l[2]:=[];
%l[2][1]:=2^2-1;
%for c in [2..t] do
%l[2][c]:=5*l[2][c-1]+10*l[1][c-1]+(5^c-1)/(5-1)*(2^2-1);
%od;
%
%for a in [3..t] do
%l[a]:=[];
%l[a][1]:=(5^1-1)/(5-1)*(2^a-1);
%for c in [2..t] do
%l[a][c]:=5*l[a][c-1]+10*l[a-1][c-1]-20*l[a-2][c-1]+(5^c-1)/(5-1)*(2^a-1);
%od;
%od;
%
%return l;
%end;

\end{document}